\documentclass[12pt,reqno,fleqn]{amsart}  
\usepackage{amsmath,amstext,amsthm,amssymb,amsxtra}
\usepackage[top=1.5in, bottom=1.5in, left=1.25in, right=1.25in]	{geometry}
\usepackage[normalem]{ulem}
\usepackage{txfonts} % pxfonts txfonts 
\usepackage[T1]{fontenc}
\usepackage{lmodern}
\usepackage{todonotes}
  
 \usepackage{euler}   % better than the option below

\usepackage{mathtools}
\mathtoolsset{showonlyrefs,showmanualtags}

\usepackage{hyperref} %,hypertexnames=false,colorlinks,[pagebackref]
\hypersetup{
    colorlinks=true,       % false: boxed links; true: colored links
    linkcolor=blue,          % color of internal links
    citecolor=magenta,        % color of links to bibliography
    filecolor=magenta,      % color of file links
    urlcolor=cyan           % color of external links
}

\usepackage[msc-links]{amsrefs}

\theoremstyle{plain} % definition 
\newtheorem{lemma}[equation]{Lemma} 
\newtheorem{proposition}[equation]{Proposition} 
\newtheorem{theorem}[equation]{Theorem} 
\newtheorem{corollary}[equation]{Corollary} 
\newtheorem{conjecture}[equation]{Conjecture}

\DeclareMathOperator{\ssup}{{\textup{sup}}}
\DeclareMathOperator{\mmax}{{\textup{max}}}
\DeclareMathOperator{\iinf}{{\textup{inf}}}

\renewcommand{\sup}{\ssup \displaylimits}
\renewcommand{\inf}{\iinf \displaylimits}
\renewcommand{\max}{\mmax \displaylimits}
\renewcommand{\log}{\textup{log}}
\renewcommand{\sin}{\textup{sin}}

\theoremstyle{definition}
\newtheorem{definition}[equation]{Definition} 

\theoremstyle{remark}

\numberwithin{equation}{section}

\title[Sparse Bounds for Walsh Multipliers] {Endpoint Sparse Bounds   for Walsh-Fourier multipliers of Marcinkiewicz type}

 \author[Wei Chen]{Wei Chen}
\address{School of Mathematical Sciences, Yangzhou University, Yangzhou 225002, China
\newline \indent School of Mathematics, Georgia Institute of Technology, Atlanta GA 30332, USA}
\email {weichen@yzu.edu.cn}
 
 \author[A.\ Culiuc]{Amalia Culiuc}
 \address{Department of Mathematics and Statistics, Amherst College, \newline  \indent
 202 Seeley Mudd Building, 31 Quadrangle
Amherst, MA 01002 USA}
\email{aculiuc@amherst.edu (A.\ Culiuc)}
 \author[F. Di Plinio]{Francesco Di Plinio} \address{\noindent Department of Mathematics, University of Virginia,  \newline \indent Kerchof Hall,  Box 400137, Charlottesville, VA 22904-4137, USA}
 \email{francesco.diplinio@virginia.edu}
\author[M.\ Lacey]{Michael Lacey}   %  can use \and  
\address{School of Mathematics, Georgia Institute of Technology, Atlanta GA 30332, USA}
\email{lacey@math.gatech.edu (M.\ Lacey)}
 \author[Y.\ Ou]{Yumeng Ou}
 \address{\noindent Department of Mathematics, City University of New York - Baruch College, \newline  \indent Box B6-230, One Bernard Baruch Way, New York, NY 10010 }
\email{yumeng.ou@cuny.baruch.edu (Y.\ Ou)} 
\thanks{Wei Chen is supported by the National Natural Science Foundation of China (11771379), the Natural Science Foundation of Jiangsu Province (BK20161326), and the Jiangsu Government Scholarship for Overseas Studies (JS-2017-228). 
Amalia Culiuc  is supported in part by the National Science Foundation under the grant   NSF-DMS-1853112. 
F.\ Di Plinio is partially
supported by the National Science Foundation under the grants   NSF-DMS-1650810 and  NSF-DMS-1800628.
M.\ Lacey's research supported in part by grant  from the US National Science Foundation, DMS-1600693 and the 
Australian Research Council ARC DP160100153. 
Yumeng Ou is supported in part by the National Science Foundation under the grant   NSF-DMS-1854148.}

\begin{document}

	\begin{abstract} We prove  endpoint-type sparse bounds for Walsh-Fourier Mar\-cin\-kie\-wicz multipliers and Littlewood-Paley square functions. These results are motivated by  work of Lerner in the Fourier setting.  As a corollary, we obtain novel quantitative weighted norm inequalities for these operators. Among these, we establish   the sharp growth rate of the $L^p$ weighted operator norm in terms of the $A_p$ characteristic in the full range $1<p<\infty$ for    Walsh-Littlewood-Paley square functions, and a restricted range for Marcinkiewicz multipliers.  
	Zygmund's $L{(\log L)^{{\frac12}}}$ inequality  is the core of our   lacunary multi-frequency projection proof.  We use the Walsh setting to avoid extra complications in the arguments. 
	\end{abstract} 

		\maketitle
		\setcounter{tocdepth}{1}
		\tableofcontents 
	
%%%%%%%%%%%%%%%%%%%%%%%%%%%%%% SECTION  SECTION SECTION
%%%%%%%%%%%%%%%%%%%%%%%%%%%%%% SECTION  SECTION SECTION 
\section{Introduction} %\label{s:}
We establish endpoint sparse bounds for   Walsh-Fourier multipliers of limited smoothness.
A recent article of Andrei Lerner \cite{180306981} poses two interesting conjectures concerning sparse bounds 
for the classical Littlewood-Paley inequality and Marcinkiewicz multipliers.  
We recall these here. For a Fourier multiplier $ m \;:\; \mathbb R \mapsto \mathbb R $, let 
\begin{equation*}
T_m f (x) = \int e ^{2 \pi i x \xi } \widehat f (\xi ) m (\xi ) \; d \xi  
\end{equation*}
be the associated linear operator.  
If $ m = \mathbf 1_{I}$, for interval $ I$, we write $ S _{I} = T _{\mathbf 1_{I}}$. 
The multiplier $ m$ is said to be \emph{Marcinkiewicz} if 
\begin{equation}\label{e:Mw}
\lVert m\rVert _{M} :=   \lVert m \rVert _{\infty } + 
\sup _{j} \lVert  m \mathbf 1_{ 2 ^{j} \leq \lvert \xi   \rvert < 2 ^{j+1} }\rVert _{BV} < \infty . 
\end{equation}
Define a  Littlewood-Paley square function by  
\begin{equation*}
(S_\lambda f) ^2 := \sum_{k \in \mathbb Z } \lvert  S _{    [ \lambda ^{k},  \lambda  ^{k+1})} f \rvert ^2 , \qquad \lambda >1. 
\end{equation*}

We are concerned with \emph{sparse bounds}, a recently active line of research  \cites{160305317,MR3625108}, which provide a stronger localized quantification of $L^p$-boundedness properties of maximal and singular integral operators.
For an interval $ I$, and index $ 0< p < \infty $, let 
\begin{equation*}
\langle  f \rangle _{I}  :=  \int _{I} \lvert  f (x)\rvert   \;  \frac{d x}{ \lvert  I\rvert }, \qquad 
\langle  f \rangle _{I,p}  := \langle  |f|^p \rangle _{I}^{\frac1p} .   
\end{equation*}  
We will also use local Orlicz norms, principally $ L(\log L)^{{\frac12}}$.  
Thus, for   convex increasing  $ \psi:[0,\infty) \to [0,\infty)  $
\begin{equation*}
\langle f \rangle _{I, \psi (L)} = \lVert  f \mathbf 1_{I}\rVert _{\psi (L) \left(I, \frac{dx}{\lvert  I\rvert }\right)}.  
\end{equation*}
The most important example for us is $ \psi _2 (x) =  \lvert  x\rvert  (\log (2 +  \lvert  x\rvert )) ^{1/2} $, which is the Orlicz function which defines $ L (\log L) ^{1/2} $. We will also reference the dual space $ \textup{exp} (L ^2 )$. 
 It is a useful remark that 
\begin{equation}\label{e:supinf}
\lVert \phi \rVert _{\textup{exp }(L ^2 )} 
\simeq \sup _{1< p < \infty } p ^{-1/2} \lVert \phi \rVert_p, \qquad 
\lVert f \rVert _{\psi _2}  \lesssim \inf _{1< q < 2} (q-1) ^{-1/2} \lVert f\rVert_q.  
\end{equation}

A collection of intervals $ \mathcal S$ are said to be \emph{sparse} if 
there is a secondary collection of pairwise disjoint 
sets $ \{F_I \;:\; I\in \mathcal S\}$ with $ F_I \subset I$,  and 
$ \lvert  F_I\rvert > c \lvert  I\rvert  $.  Here $ 0< c < 1$ is a constant. If $\mathcal S$ is a subset of a dyadic grid,  a sufficient condition for sparseness is that 
\begin{equation}
\label{e:disjset}
\left|E_I :=\bigcup\{J \in \mathcal S: J \subsetneq I\} \right| \leq (1-c)|I| \qquad \forall I \in \mathcal I.
\end{equation}
The role of the constant $ c$ is not important, and so we suppress it, though it might change from time to time in the proofs.  

%%%%%%%%%%%%%%%%%%%%%%%%%%%%%%  DEFINITION DEFINITION DEFINITION
\begin{definition}\label{d:sparse}[Sparse Bounds] 
We recall the basic notions of  $(p,q)$-sparse norms, see for instance \cite{CoCuDPOu,MR3647935}, and introduce similar concepts in the Orlicz space  setting.
\vskip 2mm \noindent
\textbf{1.} Given a sublinear operator $ T$, and indices $ 1\leq p, q \leq \infty $, we define the $(p,q)$-\emph{sparse norm} $ \lVert T \rVert _{p,q}$ to be the 
infimum over constants $ C$ so that for all bounded compactly supported $ f, g$, there holds 
\begin{equation}  \label{e:sparseDef}
\langle T f, g \rangle \leq C \sup _{\mathcal S}\sum_{I\in \mathcal S} \lvert  I\rvert 
\langle f \rangle _{I,  p} \langle g \rangle _{I,q}.   
\end{equation}
The notation $ \lVert T \rVert _{\psi _2 , q}$ will stand for the same norm, but with $ \langle f \rangle _{I, p}$ above replaced by $ \langle f\rangle _{I, \psi _2}$.  
\vskip 2mm \noindent
\textbf{2.} We define inhomogeneous sparse bounds by setting $ \lVert T \rVert _{r,p,q}$ to be the 
infimum over constants $ C$ so that for all bounded compactly supported $ f, g$, there holds 
\begin{equation}  \label{e:sparseDefDef}
\langle (T f) ^{r}, g \rangle \leq C \sup _{\mathcal S}\sum_{I\in \mathcal S} \lvert  I\rvert 
\langle f \rangle _{I,  p} ^{r} \langle g \rangle _{I,q}.   
\end{equation}
These inhomogeneous sparse bounds will be used to estimate the square function $S_\lambda$.

\end{definition}
%%%%%%%%%%%%%%%%%%%%%%%%%%%%%%  DEFINITION DEFINITION DEFINITION

%%%%%%%%%%%%%%%%%%%%%%%%%%%%%% CONJECTURE CONJECTURE CONJECTURE
\begin{conjecture}\label{j:lerner}  \cite{180306981}*{\S5.2}  The following sparse norm estimates hold true:\begin{align}\label{e:Sconj}
& \lVert S_ \lambda \rVert _{\psi _2 , 1} \lesssim 1,
\\  \label{e:Trr} &
\lVert T_m \rVert _{q,q}  \lesssim \frac {\lVert m\rVert _{M}}{\sqrt {q-1}} , \qquad 1< q < 2.  
\end{align}                                       
Here, we are using the norm of the Marcinkie\-wicz multiplier, as defined in \eqref{e:Mw}. 
\end{conjecture}
%%%%%%%%%%%%%%%%%%%%%%%%%%%%%% CONJECTURE CONJECTURE CONJECTURE

Andrei Lerner does not label these as conjectures, but suggests that they might be true: we have labeled them as conjectures for reasons of clarity. 

\smallskip 

In this paper, we are concerned with the direct analogues of Marcinkiewicz multipliers and Littlewood-Paley  square functions
in the Walsh-Fourier setting, with the formal definitions delayed to the next section.
We establish  a wide scope of sparse domination estimates, which in particular  include the Walsh analogue of  
the conjectured \eqref{e:Sconj}---\eqref{e:Trr}. 
The Walsh-Fourier setting arises naturally as a model case, as  it preserves the essential difficulties of the Fourier case without some of the purely technical difficulties proper of the latter; in particular,   Schwartz tails phenomena are absent.  
It is often the case that the right proof in the Walsh setting can be transferred to the Fourier case: this will be the object of forthcoming work.   See for instance \cites{MR2914604,MR2997005,MR2199086,DPLer2013} for closely related results.

To state and prove our main result, it is in fact convenient to   work with a more general scale of multiplier classes than the Marcinkiewicz  class. In fact, the closely related  classes $R_{p}$ have already been featured in the  characterization of the weak type endpoint behavior of Fourier-Marcinkiewicz multipliers due to Tao and Wright \cite{MR1900894}.
%%%%%%%%%%%%%%%%%%%%%%%%%%%%%%  DEFINITION DEFINITION DEFINITION
\begin{definition}\label{d:jumps}
 Let $ J\in \mathbb N$.
We say that $m\in \mathbb R^{\mathbb N}$ is  \emph{an $ R _{p,1}$-atom of at most $ J$ jumps}  if 
\begin{equation} \label{e:atom}
m (n) =  J^{- \frac{1}p} \sum_{k=1} ^{\infty }m_k (n) , \qquad m_k (n) =\sum_{j=1} ^{J_k} \mathbf 1_{\omega_{j,k}}(n) \qquad n \in \mathbb N
\end{equation}
where $ \{\omega_ {j,k}  \;:\; 1\leq j \leq J_{k}\}$ are disjoint intervals   contained in $ [ 2 ^{k-1}, 2 ^{k})$ and $ J_k\leq J$ for all $k \in \mathbb N$. 
We set $ \lVert m \rVert _{R _{q,1}}=1$.  
 We denote by $ R_{p,1}$  the atomic space generated by $ R _{p,1}$-atoms of at most $J$ jumps, with $J$ ranging over $\mathbb N$. %Therefore, we may restrict ourselves to proving Theorem \ref{t:main} (1) in the case of $m$ being an atom with at most $J$ jumps for some fixed but arbitrary $J \in \mathbb N$.
\end{definition}
%%%%%%%%%%%%%%%%%%%%%%%%%%%%%%  DEFINITION DEFINITION DEFINITION

An elementary argument shows that $ \lVert m\rVert _{M} \simeq \lVert m \rVert _{R _{1,1}}$.   We will prove this in Proposition~\ref{p:M} below. 
By virtue of this remark, the classes $R_{q,1}$ appear as the natural multiplier scale in our main Theorem below.
We are primarily interested in Walsh versions of the Lerner conjectures recalled above. 
We also prove novel sparse bounds for $ R _{q,1} $ multipliers. 
Our theorem can also be interpreted as a Walsh and sparse variant of a weak type result proved by Seeger and Tao \cite{MR1839769} and further refined by  Tao and Wright \cite{MR1900894}.

%%%%%%%%%%%%%%%%%%%%%%%%%%%%%% THEOREM THEOREM THEOREM
\begin{theorem}\label{t:main} 
%%  ENUMERATE 
The following hold.  Let $1\leq q\leq2$,  $m\in R_{q,1}$ and $T_m$ be the corresponding Walsh multiplier operator, 
and $ S _{\lambda }$ the Littlewood-Paley square function, both  defined in Section \ref{s:2} below.  Recall the notation from \eqref{e:sparseDef}.   These sparse bounds hold: 
\begin{align}
\label{eq:Sm}&
\lVert S_\lambda   \rVert_{r, \psi_2,1}  \lesssim 1  ,  \qquad 1\leq r \leq 2, 
\\ &
\label{eq:ST}
\lVert S_2\circ T_m   \rVert_{r, \psi_2,1}  \lesssim  \lVert m\rVert _{M} ,  \qquad 1\leq r \leq 2, 
\\ & \label{e:Tq}
\lVert T _{m}\rVert _{\psi _2, q}   \lesssim \tfrac{\|m\|_{R_{q,1}}  }  {\sqrt {q-1}}, \qquad 1 < q \leq 2, 
\end{align}
In the reverse direction, there is an Marcinkie\-wicz multiplier $ m$ so that 
\begin{equation}\label{e:reverse}
\lVert T _{m}\rVert _{q,q} \gtrsim \frac{1} {q-1}, \qquad 1< q < 2. 
\end{equation}

\end{theorem} 
 
%%%%%%%%%%%%%%%%%%%%%%%%%%%%%% THEOREM THEOREM THEOREM
 
Notice that the sparse bounds \eqref{eq:Sm} include  the conjecture of Lerner ($ r=1$), but there are a full range of interesting sparse bounds for the square function. 
Taking  $ r = \frac{3}2$  will yield the sharp $ A_p$ inequalities for the square function. 
The bound for $ S _{2} \circ T_m$ is an inhomogeneous sparse bound for the composition of the Haar square function $ S_2$ composed with a Marcinkie\-wicz multiplier.  
We were inspired to seek for such sparse estimate in light of Lerner's arguments \cite{180306981}.   

The conjectured inequality \eqref{e:Trr} for multipliers is however false. The correct version of the inequality is 
\begin{equation} \label{e:q-1}
\lVert T_m \rVert _{q,q}  \lesssim \frac {\lVert m\rVert _{M}}{ {q-1}} , \qquad 1< q < 2.
\end{equation}
This is a corollary to \eqref{e:Tq} and \eqref{e:supinf}.  And the rate of growth in $ q$ is sharp, in view of \eqref{e:reverse}.

The method of proof of the sparse bounds depends upon a multi-frequency decomposition of the multipliers. 
The core of this argument is in \S \ref{s:core}.  It is the main point of interest in this paper.  
The proof of Theorem \ref{t:main} descends from this decomposition via iterative arguments: in particular \S \ref{s:4} contains the proof of \eqref{e:Tq}, \S \ref{s:6} is devoted to the square function bounds \eqref{eq:Sm} and \eqref{eq:ST}. 
The lower bound \eqref{e:reverse} is described in \S \ref{s:lower}.

\bigskip 

As customary in the subject,  quantitative weighted norm inequalities descend from sparse norm estimates. We detail the   consequences of the estimates from Theorem \ref{t:main} in the next corollary: compare with the discussion in Lerner \cite{180306981}.  We use the language of $ A_p$ weights and postpone  definitions and proofs to Section  \ref{s:weight}.

%%%%%%%%%%%%%%%%%%%%%%%%%%%%%% COROLLARY COROLLARY COROLLARY
\begin{corollary}\label{c:lerner} For the operators $ T_m$ and $ S _{\lambda }$ on the Walsh system, and weights $ w$, there holds  for $ 1<q < 2$, 
\begin{align}    &\label{e:cP2}
\lVert T _{m}\rVert _{L ^{1}(\log L ) ^{1/2} (w)  \to L ^{1, \infty } (w)} < C _{[w] _{A_1}} \lVert m\rVert _{M}, 
\\  &\label{e:cLp}
\lVert T _{m}\rVert _{L ^{p} (w) \to L ^{p} (w)} \lesssim \lVert m \rVert _{M} [w] _{A_p} ^{\frac{3}2 \max \{1, (p-1) ^{-1}\}}, \qquad  \max\{p, p'\}\geq \tfrac{5}2 , 
\\ & \label{e:Tqq}
\lVert T _{m}\rVert _{L ^{p}  (w) \to L ^{p} (w)} <  C_{[w] _{ A_ {p/q}}, [w] _{RH _{ (q'/p)'}} }  \lVert m\rVert _{R _{q,1}}
\qquad   q < p < q', 
\\ &\label{e:cS}
\lVert  S _{\lambda }\rVert _{L ^{p} (w) \to L ^{p} (w)} \lesssim   [w] _{A_p} ^{  \max \bigl\{1 , \frac{3} {2 (p-1)}\bigr\}}  , \qquad 1< p < \infty . 
\end{align}
\end{corollary}
%%%%%%%%%%%%%%%%%%%%%%%%%%%%%%  COROLLARY COROLLARY COROLLARY

The square function inequalities  \eqref{e:cS} for $ 1< p < 2$ were established by Lerner \cite{180306981} 
in the Fourier case.   
Otherwise, the inequalities above are new, and sharp in the power of the $ A_p $ characteristic for the square function \eqref{e:cS} and the Marcinkie\-wicz multipliers \eqref{e:cLp}, as conjectured by Lerner.  But notice that in \eqref{e:cLp}, we require $ 1< p \leq  \tfrac{5}3$ or $ p \geq \tfrac{5}2$. 
We leave open the sharp dependence on $ A_p$ for $ \tfrac{5}3 < p < \tfrac{5}2$. 

The other inequalities complement un-weighted estimates of Tao and Wright \cite{MR1900894}.  All the estimates above can be made quantitative and, likewise, it is potentially interesting to detail weighted weak type estimates; we do not pursue these points in great detail here. 

We remark that the Fourier and Walsh cases can, in certain instances, diverge. For instance, the Square Function inequalities above 
in the case of $ \lambda =2$ are trivial, as in that case, the Square function is in fact the Haar, or dyadic martingale, square function. Much stronger inequalities are true in that case.  
\subsection*{Acknowledgments} This project was initiated during F.\ Di Plinio and Y.\ Ou's Spring 2018 visit to the Georgia Tech Mathematics Department, whose hospitality is gratefully acknowledged.
 The authors are grateful to Andrei Lerner for his insightful comments on the sparse estimates of Theorem \ref{t:main}.

%%%%%%%%%%%%%%%%%%%%%%%%%%%%%% SECTION  SECTION SECTION
%%%%%%%%%%%%%%%%%%%%%%%%%%%%%% SECTION  SECTION SECTION 
\section{Walsh Analysis} \label{s:2}

In this section, we define the Walsh functions as well as  the Walsh analogues of the Fourier multiplier operators described in the introduction, and state some definitions that we will need for the proof.

\subsection{The Walsh system and Walsh multipliers}  The Walsh functions are the group characters of $ [0,1] \equiv \{0,1\} ^{\mathbb N }$, where $ \mathbb N = \{0,1, 2 , \ldots  \}$.    

To be explicit, define first the Walsh functions $ \{w _{2 ^{k}} \;:\; k=0,1 ,\ldots \}$
by 
\begin{equation*}
w _{2 ^{k}} (x) = \textup{sign} (\sin (2 ^{k+1} \pi x )), \qquad 0< x < 1. 
\end{equation*}
Then, $ \{w _{2 ^{k}} \}$ forms a Rademacher sequence.  
In addition, set  $ w _0 \equiv 1$.  
Extend this to all integers $ n \in \mathbb N $ by writing $ n = 2 ^{k_1} + \cdots + 2 ^{k_d}$ uniquely as a sum of distinct powers of $ 2$, and then define 
\begin{equation*}
w _{n} := w _{2 ^{k_1}} \cdot \cdots \cdot w _{2 ^{k_d}}. 
\end{equation*}
They satisfy this \emph{restricted product rule}: $ w _{m+n} = w _{m} w _{n}$ if $ n,m$ do not have a common non-zero binary digit. 
The Walsh functions  form an orthogonal basis for 
$ L ^{2} (0,1)$.  They are a discrete variant of the exponentials $ \{e ^{2 \pi i k x} \;:\; k = 0 ,1 ,\ldots \}$, and so we will write 
\begin{equation*}
f = \sum_{n=0} ^{\infty } \widehat f (n) w_n , \qquad \widehat f (n) = \langle f, w_n \rangle.   
\end{equation*}
The reader can note that several properties of Walsh functions recalled below relate to precise Fourier localization, and have proper analog in the Fourier basis.  

To a bounded function $m:  \mathbb N \to \mathbb R $ we associate a  \emph{Walsh multiplier} by 
\begin{equation*}
T _{m} f = \sum_{n \in \mathbb N } m (n) \langle f, w _{n} \rangle w _n .  
\end{equation*}
If $ m = \mathbf 1_{\omega } $, for an interval $ \omega $, we understand that 
$ \omega = [\alpha , \beta )$ or $ \omega = [\alpha ,\beta ]$, for integers $ \alpha , \beta \geq 0$, and write 
  $ T_ {\mathbf 1_{\omega }} = T _{\omega }$.  
The Walsh-Littlewood-Paley square function with integer parameter $\lambda \geq 2$ is then  defined  by
\begin{equation}
\label{e:lam}
S_\lambda f(x)^2= \lvert  \widehat f (0) \rvert ^2 + \sum_{k\in \mathbb N} |T_{[\lambda^{k-1},\lambda^{k})} f(x)|^2.
\end{equation}
Note that $ S_2$ is the usual Haar square function.

Each Rademacher function $ w _{2 ^{k}}$ is constant on dyadic intervals of length $ 2 ^{-k-1}$. 
From this it follows that  if $ n < \lvert  I\rvert ^{-1}  $, for dyadic interval $ I$, then $ w _{n}$ is constant on $ I$. 
Hence,  $ T _{ [0,n]}f$ is constant on $ I$.  But also note that 
\begin{equation} \label{e:CE} 
T _{ [0,2 ^{k})}f = \sum_{I \;:\; \lvert  I\rvert = 2 ^{-k} } \langle f \rangle_I \mathbf 1_{I}. 
\end{equation}
So that in this case we have the stronger localization property 
\begin{equation}\label{e:00}
\mathbf 1_{I}  T _{[0,2 ^{k})} (f \mathbf 1_{[0,1] \setminus I}) \equiv 0, \qquad \lvert  I\rvert = 2 ^{-k}.    
\end{equation}
This principle is an important fact for us.  

\subsection{Tiles} 

We will be discretizing Walsh multipliers by means of \emph{Walsh wave packets}, namely localizations of the Walsh characters to a dyadic spatial interval. The language of \emph{tiles} will describe the phase space regions associated to wave packets. 

Let  $ \mathcal D _0$ be the standard dyadic grid on $ [0,1]$, 
and $ \Omega _ 0 $ be the collection of dyadic subintervals $\omega \subset (0,\infty)$ of length $|\omega|\geq 1$.  We say that $ p = I_p \times \omega _p \in \mathcal D_0 \times \mathcal \Omega_0 $ is a \emph{tile} if $ \lvert  I_p\rvert \cdot \lvert  \omega_p \rvert =1  $.  In that case, we have 
\[
\omega _p = \textstyle \left[\frac{n}{\lvert  I_p\rvert} , \frac{n+1}{\lvert  I_p\rvert}\right) , \qquad n\in \mathbb N  
\]
and we define the \emph{wave packet associated to $ p$} to be 
\begin{equation*}
w _{p} (x) = \frac{1}{\lvert  I_p\rvert ^{1/2}} w_n \Bigl( \frac{x - \ell_{I_p}} {\lvert  I_p\rvert } \Bigr) \mathbf 1_{I_p} (x), 
\end{equation*}
where $ \ell_{I_p}$ is the left endpoint of $ I_p$.   Under this definition, it follows that for all intervals $ I\in \mathcal D_0$
\begin{equation*}
\mathbf P (I) := \{  w _p \;:\; I_P = I\}
\end{equation*}
is an orthonormal basis for $ L ^2 (I)$.  (In particular, if $ I=[0,1]$, we recover the Walsh basis.) 
A much deeper property \cite{MR2199086} is the following orthgonality property:  For any two tiles $ p, q$
\begin{equation*}
\langle w _p , w _{q} \rangle =0  \quad \textup{if and only if} \quad   p\cap q = \emptyset . 
\end{equation*}
We understand the intersection to be of two rectangles in $ [0,1) \times (0, \infty )$.  
This property leads to many simplifications in the Walsh case.

%%%%%%%%%%%%%%%%%%%%%%%%%%%%%% SUBSECTION SUBSECTION SUBSECTION SUBSECTION
 %%%%%%%%%%%%%%%%%%%%%%%%%%%%%% SUBSECTION SUBSECTION SUBSECTION SUBSECTION 
\subsection{Tiles and Haar Functions}%\label{ss.}
The Haar basis can be described as tile basis.  Namely, the Haar basis on $ L ^2 (0,1
)$ can be given as follows.  Set $ \chi _{[0,1)} = w _{0} = w _{[0,1] \times [0,1]}$, 
and for dyadic $ I\subset [0,1]$, observe that the classical Haar function is given by 
\begin{equation*}
h _{I} =  \frac {\mathbf 1_{I _{-}  } - \mathbf 1_{I _{+}}} {\sqrt {\lvert  I\rvert }} 
= w _{I\times\frac{1}{|I|}[1,2)  }.  
\end{equation*}
The tiles for the Haar system partition $ [0,1] \times [0, \infty )$, and 
the tiles for the Walsh and Haar systems can be visualized as in Figure~\ref{f:WH}. 

%%%%%%%%%%%%%%% Figure                                            
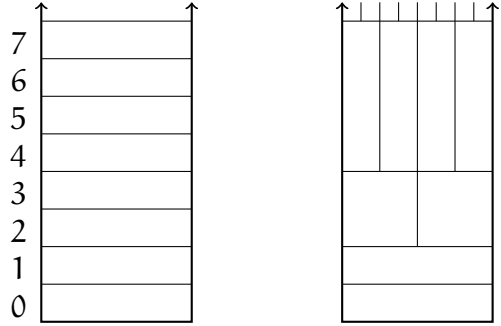
\begin{figure}[h!]
\centering
\begin{tikzpicture}[yscale=.5] 
\draw[<->,thick] (0, 8.5) -- (0,0) --  (2,0) -- (2,8.5); 
\foreach \x in {1,2,...,8} \draw (0,\x) -- (2,\x);
\foreach \x in {0,1,...,7} \draw (-.3,\x+1) node[below] {$    \x$};  

\begin{scope}[xshift=4cm] 
\draw[<->,thick] (0, 8.5) -- (0,0) --  (2,0) -- (2,8.5); 
\foreach \x in {1,2,4,8} \draw (0,\x) -- (2,\x); 
\draw (1,2) -- (1,4); 
\foreach \x in {.5,1,1.5} \draw (\x,4) -- (\x,8);
\foreach \x in {.25,.5,.75,1,1.25,1.5,1.75} \draw (\x,8) -- (\x,8.5);
\end{scope} 
\end{tikzpicture}
\caption{The tiles associated to the Walsh basis on the left, and the Haar basis on the right.
Note that the regions $ [0,1] \times [0, 2 ^{k})$ are tiled by both the Walsh and the Haar tiles, which is a reflection of the identity \eqref{e:CE}. 
} 
\label{f:WH}
\end{figure}
%%%%%%%%%%%%%%% Figure   

We record another useful fact, one that is basic to the analysis of our multipliers. 
Each Walsh projection $ T _{[0, n )}$ has an expansion in terms of tile operators.  Write $ n$ in binary, namely $n = \sum_{j=1} ^{k} 2 ^{n_j}$, for $ n_1 >  \cdots > n_k\geq 0$; we use \emph{decreasing order}.  
Let $ n ^{0}=0$,  $ n ^{1} = 2 ^{n_1} ,\dotsc,  n ^{k}=n$ be the partial sums of the binary expansion of $ n$, 
and let 
\begin{equation*}
\omega ^j = [n ^{j-1} , n ^{j}) ,\qquad   j=1 ,\dotsc, k. 
\end{equation*}
Each of these intervals are dyadic.  
Here are three examples:  For $ [0, 2 ^{t})$ we have $ \omega ^{1} = [0, 2 ^{t})$. For  $ [0,5)$, we have  $ 5 = 2 ^{2}+ 2 ^{0}$, so that 
$ \omega ^{1} = [0,4)$ and $ \omega ^2 = [4,5)$.  For $ [0,11)$, we have  $ 11 = 2 ^{3} + 2 ^{1} + 2 ^{0}$, so that 
$ \omega ^{1} = [0,8)$, $ \omega ^{2} = [8,10)$, and $ \omega ^{3} = [10,11)$, see Figure~\ref{f:WalshTile}.  

Setting  $ \mathbf{P}(n) = \{p\in \mathbf P:\omega_p=\omega ^j \textup{ for some } 1\leq j \leq k\}$, we have 
\begin{equation} \label{e:WalshTile}
T _{[0,n)}f = \sum_{p \in \mathbf{P}(n)} \langle f, w _p \rangle w_p. 
\end{equation}
This is illustrated in Figure~\ref{f:WalshTile}.   But, we have this further property, which follows from the definitions and routine computation: see  \cite[Lemma 2.1]{MR3084406} for the details.  For $ p = I \times [n ^{j-1}, n ^{j}) \in \Omega $, we have 
\begin{equation}\label{e:signWalsh}
w _{n} w _{p} = \sigma h _{I}, \qquad \sigma \in \{\pm1\}.  
\end{equation}
That is, up to a modulation, the tiles in \eqref{e:WalshTile} are in fact Haar functions.

%%%%%%%%%%%%%%% Figure
\begin{figure}[t]\centering
\begin{tikzpicture}[yscale=.5] 
\draw[<->,thick] (0, 11.5) -- (0,0) --  (2,0) -- (2,11.5); 
\foreach \x in {0,1,...,11} \draw (0,\x) -- (2,\x);
\foreach \x in {0,1,...,11} \draw (-.3,\x+0.5) node[below] {$ {}_{\x}$};  

\begin{scope}[xshift=4cm] 
\draw[<->,thick] (0, 11.5) -- (0,0) --  (2,0) -- (2,11.5); 
 \foreach \x in {.25,.5,.75,1,1.25,1.5,1.75} \draw (\x,0) -- (\x,8);
 \foreach \y in {8 ,10 ,11}\draw (0,\y) -- (2,\y); 
  \draw (1,8) -- (1,10) ; 
 \foreach \w/\y/\a/\b in {1/4/0/8, 2/9/8/10, 3/10.5/10/11}  \draw (2.9,\y) node {$ {}_{\omega ^{\w} =[\a,\b)}$};  
\end{scope} 
\end{tikzpicture}
\caption{The Walsh Projection $  T_{[0,11)}$ as a sum of tiles, in two different ways, illustrating \eqref{e:WalshTile}. 
Note that $ 11 = 8 + 2 + 1$, and that the corresponding $ \omega ^{j}$ are indicated on the right.  
} 
\label{f:WalshTile}
\end{figure}
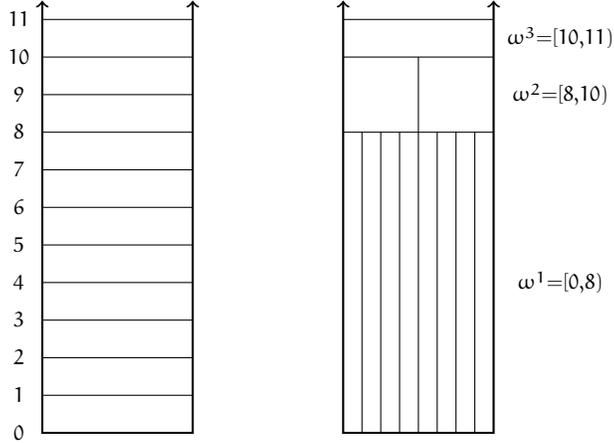
%%%%%%%%%%%%%%% Figure

 %%%%%%%%%%%%%%%%%%%%%%%%%%%%%% SUBSECTION SUBSECTION SUBSECTION SUBSECTION 
\subsection{Zygmund's Inequality}%\label{ss.}

A basic example of a Marcinkie\-wicz multiplier is 
\begin{equation*}
m (n) =\sum_{k=1} ^{\infty } \varepsilon _k \mathbf 1_{ \{\lambda _k\}} (n) ,  \qquad \varepsilon _k \in \{\pm 1\}.  
\end{equation*}
where $ \{\lambda _k\}$ is an increasing sequence of integers with $ \inf \{ \lambda _{k+1}/ \lambda _k\} >1$.  We refer to such sequences  $ \lambda _k$ as \emph{lacunary}.  
Observe that the uniform control of these multipliers from $ \Psi (L) $ to weak $ L ^{1}$ implies the stronger inequality 
\begin{equation*}
\lVert  \{\widehat f (\lambda _k) \}\rVert _{\ell ^2 } \lesssim \lVert f\rVert _{\Psi (L)}. 
\end{equation*}
Indeed, the core of our proof is to take the endpoint version of the above inequality, and lift it to the setting of more general multipliers.  This  endpoint version is  known as Zygmund's inequality.  

%%%%%%%%%%%%%%%%%%%%%%%%%%%%%% THEOREM THEOREM THEOREM
\begin{theorem}\label{t:z}  Let $ \{\lambda _k\}$ be lacunary integers. There holds 
\begin{equation}\label{e:Z}
\lVert  \{\widehat f (\lambda _k) \}\rVert _{\ell ^2 }\lesssim \lVert f\rVert _{L (\log L) ^{{\frac12}} }.
%\\ \label{e:ZH}
%\lVert  \{\widehat f (\lambda _k) \}\rVert _{\ell ^2 } &\lesssim \lVert f\rVert  _{H ^{p}}, \qquad 0< p  \leq 1  . 
\end{equation}
\end{theorem}
%%%%%%%%%%%%%%%%%%%%%%%%%%%%%% THEOREM THEOREM THEOREM

The right hand side of inequality \eqref{e:Z} is sharp. Since the left hand side is associated to Marcinkie\-wicz multipliers, it follows that the sparse bounds that we prove are sharp. 
 Inequalities such as that of Theorem \ref{t:z} are basic to Tao and Wright \cite{MR1900894}, and have  found previous usage in the treatment of pointwise convergence of lacunary Walsh series near $L^1$, see \cites{MR3189278,MR2914604}.
Closely related is a version of the Chang-Wilson-Wolff inequality \cite{MR800004}. 
%%%%%%%%%%%%%%%%%%%%%%%%%%%%%% THEOREM THEOREM THEOREM
\begin{theorem}\label{t:cww} Let $ f $ have  integral zero on $[0,1]$. We have 
\begin{equation}\label{e:cww}
\lVert f \rVert _{\textup{exp} (L ^{2 })} \lesssim \lVert S (f)\rVert _{\infty } .  
\end{equation}

\end{theorem}
%%%%%%%%%%%%%%%%%%%%%%%%%%%%%% THEOREM THEOREM THEOREM

 %%%%%%%%%%%%%%%%%%%%%%%%%%%%%% PROOF PROOF PROOF
\begin{proof}[Proof of Theorem~\ref{t:z}]
The Zygmund inequality is well known, and classical. We briefly indicate a proof by duality. 
The dual to \eqref{e:Z} is 
\begin{equation*}
\Bigl\lVert \sum_{k} c_k w _{\lambda _k} \Bigr\rVert _{\textup{exp} (L ^2 )} \lesssim 
\lVert \{c_k\}\rVert _{\ell ^2 }. 
\end{equation*}
As the square function $ S$ of the function on the left is pointwise dominated by $ \lVert \{c_k\}\rVert _{\ell ^2 }$,  the last display follows from an application of the  Chang-Wilson-Wolff inequality \eqref{e:cww}.  
%
%Duality for \eqref{e:ZH} is more complicated, so we use a direct proof.   
%Given $ \{\lambda _k\}$, define $ j_k \in \mathbb N $ by $ 2 ^{j_k } \leq \lambda _k < 2 ^{j_k+1}$.  
%
%
%The dual to \eqref{e:ZH} is 
%\begin{equation*}
%\Bigl\lVert \sum_{k} c_k w _{\lambda _k} \Bigr\rVert _{BMO} \lesssim 
%\lVert \{c_k\}\rVert _{\ell ^2 }. 
%\end{equation*}
%But, this is a routine computation.  
\end{proof}
%%%%%%%%%%%%%%%%%%%%%%%%%%%%%% PROOF PROOF PROOF
 
%%%%%%%%%%%%%%%%%%%%%%%%%%%%%% SUBSECTION SUBSECTION SUBSECTION SUBSECTION
 %%%%%%%%%%%%%%%%%%%%%%%%%%%%%% SUBSECTION SUBSECTION SUBSECTION SUBSECTION 
\subsection{Marcinkie\-wicz Multipliers}%\label{ss.}

We prove Marcinkie\-wicz multipliers are also $ R _{1,1}$ multipliers.  
%%%%%%%%%%%%%%%%%%%%%%%%%%%%%% PROPOSITION PROPOSITION PROPOSITION
\begin{proposition}\label{p:M} For a multiplier $ m$ we have $ \lVert m \rVert _{ M} \simeq \lVert m\rVert _{R _{1,1}}$.

\end{proposition}
%%%%%%%%%%%%%%%%%%%%%%%%%%%%%% PROPOSITION PROPOSITION PROPOSITION

The core of this argument is this Lemma. The point of it is that  in \eqref{e:m}, the integrand is an indicator, and that the integration measure is 
Lebesgue, hence independent of $ m$. 

%%%%%%%%%%%%%%%%%%%%%%%%%%%%%% LEMMA LEMMA LEMMA
\begin{lemma}\label{l:m} Let $ m $ be an increasing function on $ [0,1]$ with $ m (0)=0$, and $ m (1)  < \infty  $.  
Then,  
\begin{equation}  \label{e:m}
m (\xi ) = \int _{0} ^{m (1)} \mathbf 1_{ [m ^{-1} (\theta ), m (1))} (\xi ) \; d \theta ,   \qquad 0\leq \xi < 1. 
\end{equation}

\end{lemma}
%%%%%%%%%%%%%%%%%%%%%%%%%%%%%% LEMMA LEMMA LEMMA

%%%%%%%%%%%%%%%%%%%%%%%%%%%%%% PROOF PROOF PROOF
\begin{proof}
We have 
\begin{equation*}
m (\xi )   = \int _{0} ^{1} \mathbf 1_{ [ \theta , 1]} (\xi )d m (\theta ) . 
\end{equation*}
Make the change of variables $ m(\theta) = t$.  
\end{proof}
%%%%%%%%%%%%%%%%%%%%%%%%%%%%%% PROOF PROOF PROOF

%%%%%%%%%%%%%%%%%%%%%%%%%%%%%% PROOF PROOF PROOF
\begin{proof}[Proof of Proposition~\ref{p:M}]  
It is clear that for an $ R _{1,1}$ atom $ m$ we have $ \lVert m\rVert _{R _{1,1}} \lesssim \lVert m\rVert _{M}$. And, so this inequality holds for any $ R _{1,1}$ multiplier. 

For the reverse, we argue that  a Marcinkie\-wicz multiplier $ m$ is a convex combination of $ R _{1,1}$ atoms. 
For $ j \geq 0$, the function  $ m \mathbf 1_{ [ 2 ^{j}, 2 ^{j+1})}$    
is a function of bounded variation. Such functions are the difference of two increasing functions, so we assume 
that $ m$ is increasing on each interval $ [ 2 ^{j}, 2 ^{j+1})$.  A mulitiplier that is constant on each such interval is an $ R _{1,1}$ multiplier, 
so we can assume that $ m (2 ^{j})=0$ for all $ j\geq 0$.  And, that $ \sup _{2 ^{j} \leq \xi < 2 ^{j+1}} m (\xi ) \leq 1$ for all $ j$. 

It follows from \eqref{e:m} that we can choose functions $ \tau _j$ so that \emph{uniformly in $ j$} we have 
\begin{equation*}
m (\xi ) = \int _{0} ^{1} \mathbf 1_{  [\tau _j (\theta  ),1)} ((\xi - 2 ^{j})/2 ^{j} )  \; d \theta , \qquad  2 ^{j} \leq \xi < 2  ^{j+1} . 
\end{equation*}
For each $0< \theta <1 $, the multiplier below is an $ R _{1,1}$ atom. 
\begin{equation*}
\sum_{j \geq 0}    \mathbf 1_{  [\tau _j (\theta  ),1)} ((\xi - 2 ^{j})/2 ^{j} )
\end{equation*}
It follows that $ m$ is in the convex hull of $ R _{1,1}$ atoms.  
\end{proof}
%%%%%%%%%%%%%%%%%%%%%%%%%%%%%% PROOF PROOF PROOF

%%%%%%%%%%%%%%%%%%%%%%%%%%%%%% SECTION  SECTION SECTION
%%%%%%%%%%%%%%%%%%%%%%%%%%%%%% SECTION  SECTION SECTION 
\section{The Key Decomposition} \label{s:core}
The main step in our proof of Theorem \ref{t:main} is a multi-frequency type Calder\'on-Zygmund decomposition adapted to the \emph{jumps} of the multiplier $m \in R_{p,1}$.  We detail one version of it here, and later two variants on the theme.

In the main Lemma \ref{l:key} below, we will be modifying the multiplier $m$.  
That step requires these definitions.  Given intervals $ \omega $ and $ I$, 
we set the \emph{interior of $ \omega $ relative to scale $I$} to be 
\begin{equation} \label{e:omegaI}
\omega ^{o} _{I} = \begin{cases}
\varnothing & \lvert  \omega \rvert  \lvert  I\rvert <1 
\\
{\displaystyle\bigcup} \left\{  \left[\frac{n}{\lvert  I\rvert }, \frac{n+1}{\lvert  I\rvert }\right)  \;:\;    \left[\frac{n}{\lvert  I\rvert }, \frac{n+1}{\lvert  I\rvert }\right)\subset \omega \right\} & \textsf{otherwise}
\end{cases} 
\end{equation}  
Given a multiplier $ m $ of the form 
\begin{equation}  \label{e:msum}
m = \sum_{ \omega \in \Omega _{m}} c_\omega \mathbf 1_{\omega } 
\end{equation}
such as  in the case of $m$ being an   atom of at most  $J$ jumps, 
and a dyadic interval $ I\subset[0,1]$, we set 
\begin{equation} \label{e:mI}
m _{I} =  \sum_{ \substack{ \omega \in \Omega _{m} }} c_\omega  \mathbf 1_{\omega ^{o}_I }. 
\end{equation}
In words, this is \emph{the induced multiplier on $ I$.} 
Note that for functions $ f$ supported on $ I$ 
\begin{equation*}
T _{m _{I}} f =  \sum_{p\in \mathbf P (I)} m_I (\omega _p)    \langle f, w_p \rangle w_p, 
\end{equation*}
where $m_I (\omega _p) $ is the unique value that $ m _{I} $ takes on the interval $ \omega _p$: compare to the localization principle \eqref{e:00}. 
More generally, if $I_0$ is a dyadic subinterval of $[0,1]$, the multiplier  $m$ is said to be \emph{adapted} to $I_0$ if it is  of the type \eqref{e:msum} with 
$$
\lvert\omega\rvert  \cdot  \lvert  I_0\rvert \in \mathbb N   \qquad \textup{for all $ $} \omega \in \Omega_m. 
$$
The latter display signifies that $m$ is constant on dyadic subintervals of $(0,\infty)$ of length  $\frac{1}{|I_0|}$.
The following proposition, which will be used in our recursive construction in Section \ref{s:4}, is an immediate consequence of the definition of atoms.
%%%%%%%%%%%%%%%%%%%%%%%%%%%%%% PROPOSITION PROPOSITION PROPOSITION
\begin{proposition}\label{p:induced} Let $ m $ be an atom with at most $ J$ jumps.   
For any interval $ I\subset [0,1]$ the induced multiplier  $ m_I$ is also an atom of at most $ J$ 
jumps and it is adapted to $I$.
\end{proposition}
%%%%%%%%%%%%%%%%%%%%%%%%%%%%%% PROPOSITION PROPOSITION PROPOSITION

%%%%%%%%%%%%%%% Figure
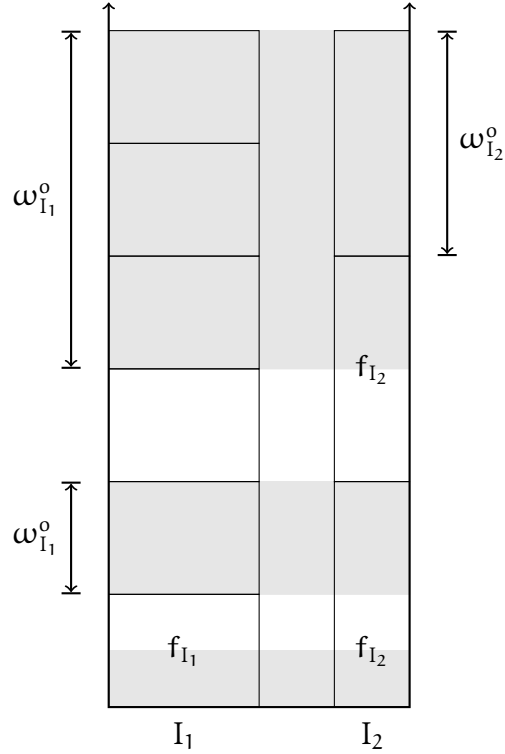
\begin{figure}[h]
\centering 
\begin{tikzpicture}[yscale=.75] 
\foreach \b/\d in {0/1,2/4, 6/8, 8/12} 
\filldraw[gray!20!white]  (0,\b) rectangle (4,\d); 

\draw[<->,thick] (0, 12.5) -- (0,0) --  (4,0) -- (4,12.5); 
%\foreach \x in {0, 1,2,4,6,8,12} \draw (0,\x) -- (-.5,\x) node[left]  { $ \x$};   

\foreach \a/\b in {0/2,2/4,4/6,6/8,8/10,10/12 }  \draw (0, \a) rectangle (2,\b); 
\foreach \a/\b in {0/4,4/8,8/12}  \draw (3, \a) rectangle (4,\b); 

\draw (1,-.5) node {$ I_1$};  \draw ( 3.5, -.5) node {$ I_2$};  

\draw[|<->|,thick]  (4.5,8) -- (4.5,12) node[midway,right ] {$ \omega ^{o} _{I_2}$}; 
\draw[|<->|,thick]  (-.5,6) -- (-.5,12) node[midway, left] {$ \omega ^{o} _{I_1}$};    
\draw[|<->|,thick]  (-.5,2) -- (-.5,4) node[midway, left] {$ \omega ^{o} _{I_1}$}; 
\draw (1,1) node {$ f _{I_1}$};
\draw (3.5,1) node {$ f_{I_2}$};\draw (3.5,6) node {$ f _{I_2}$};
\end{tikzpicture}

\caption{An illustration of the multi-frequency decomposition for a multiplier adapted to $ I_0$. The multiplier is a sum of indicators, denoted by the gray bands. 
There are two subintervals, $ I_1$ which is $ {\frac12}$ the length of $ I_0$, on the left, and $ I_2$ which is $ \tfrac 14$ the length.
  The tiles in $ \mathbf P (I_j) $, for $ j=1,2$, are drawn.  The non-empty ``interiors'' are indicated for both intervals. For $ I_1$, two tiles would contribute to $ f _{I_1}$, for $ I_2$ there is one  such tiles.  }
\label{f:multi}
\end{figure}
%%%%%%%%%%%%%%% Figure

The key multi-frequency decomposition argument is contained in the following Lemma.
%%%%%%%%%%%%%%%%%%%%%%%%%%%%%% LEMMA LEMMA LEMMA
\begin{lemma}\label{l:key} Let $ f $ be supported on  $ I_0\in \mathcal D_0$, 
and let $ m $ be an  $ R _{p,1}$-atom of at most  $J$ jumps, adapted to $ I_0$. 
Let $I \in \mathcal I $ be a collection of disjoint dyadic subintervals  of $I_0$ with 
\begin{equation} \label{e:condition}
\left(\sup_{ I_0\setminus\left( \bigcup_{I\in \mathcal I} I\right)} |f| \right) + \left(
\sup_{I \in \mathcal I } \langle f \rangle _{I, \psi_2}  \right) \leq C \langle f \rangle _{I_0, \psi_2} 
\end{equation}
for some given $C>0$. Then there exists a decomposition of $f$ with the following properties:
\begin{align}\label{e:gb} &
f =  f _{\infty } + \sum_{I\in \mathcal I} (f_I + f'_I) ,
\\ & \label{e:fzI}
\lVert  f _{\infty }\rVert_ \infty \lesssim \langle f \rangle _{I_0, \psi_2},
\\ &
 \textup{supp}\, f _I, \,\textup{supp}\, f'_I \subset I, \qquad I \in \mathcal I, \label{e:supp}
\\ &
\label{e:fI}
\langle f_I \rangle _{I,2} \lesssim  \sqrt J   \cdot \langle f \rangle _{I_0, \psi_2} , \qquad I \in \mathcal I, 
\\& \label{e:SfI}
\langle  S_I (f_I) \rangle _{I, \infty } 
\lesssim   \sqrt J   \cdot \langle f \rangle _{I_0, \psi_2} , \qquad I \in \mathcal I, 
\\ & \label{e:Th}
T _{m} f'_I = T _{m_I} f'_I ,\qquad I \in \mathcal I.  
\end{align}
In \eqref{e:SfI}, we are using the notation  
\begin{equation}
\label{e:SI}
S_I (f) ^2 := \sum_{Q \;:\; Q\subset I}  \frac{\lvert \langle f, h_Q  \rangle \rvert ^2 } {\lvert  Q\rvert } \mathbf 1_{Q}. 
\end{equation}
\end{lemma} 

The full decomposition above will be used, but also note that 
for $ \tilde f =f _{\infty } + \sum_{I\in \mathcal I} f_I$, we have 
\begin{equation}\label{e:ftilde}
\langle \tilde f  \rangle _{I_0, 2}\lesssim  \sqrt J   \cdot \langle f \rangle _{I_0, \psi_2}
\end{equation}
which is a direct consequence of \eqref{e:fzI} and \eqref{e:fI}.

%%%%%%%%%%%%%%%%%%%%%%%%%%%%%% PROOF PROOF PROOF
\begin{proof}
See Figure~\ref{f:multi} for an illustration of the different elements of the decomposition introduced here.   
We can assume that $ \langle f \rangle _{I_0, \psi_2} =1$. 
Let $ E= \bigcup \{I\in \mathcal I\} $, and $ f _{\infty } = f \mathbf 1_{I_0 \setminus E}$. 
It is clear that \eqref{e:fzI} holds.  
Our focus is on  the orthogonal decomposition of $f\mathbf{1}_I$, for $ I\in \mathcal I$,  induced by the disjoint partition of $\mathbf{P}(I)$ into 
\begin{equation} \label{e:partP}
\begin{split} 
 & \mathbf{P}_{m}  (I)=  \mathbf P (I) \setminus \mathbf{P}_{m}'(I) . 
 \\
&\mathbf{P}_{m} '(I)=\big\{p \in \mathbf P (I):   \forall  \omega \in \Omega_m:  \omega _p\subset \omega^\circ_{I} \quad \textup{$ $ or $ $ }   \omega _p \cap \omega = \varnothing \big\}. 
\end{split}  
\end{equation}
Recall the notation \eqref{e:msum} for the multiplier, and \eqref{e:omegaI} for $\omega_{I}^{\circ}$.
Notice that, in particular, $\mathbf{P}_{m}(I)$ is made of those elements of $\mathbf{P}(I)$ on which the multiplier $m$ has a  jump on the interval $ \omega _p$.  
We split $f\mathbf{1}_I =  f_I+f'_I$ with
\begin{equation} \label{e:split}\begin{split}
 & f_I=\sum_{p\in \mathbf{P}_{m} (I)} \langle f,w_p \rangle w_p,    \\
&f'_I=\sum_{p\in \mathbf{P}_{m}'(I)} \langle f,w_p \rangle w_p.
\end{split}\end{equation}

The support property \eqref{e:supp} is obvious, as is 
Property \eqref{e:Th}. 
It remains to prove the essential control on $ f_I$, \eqref{e:fI} and \eqref{e:SfI}.  Indeed, \eqref{e:fI} is a consequence of the Zygmund inequality. 
Fix a dyadic interval $ D (I,k ) = \lvert  I\rvert ^{-1}  [2 ^{k}, 2 ^{k+1}) $.   From the above discussion, the cardinality of   
\begin{equation*}
\mathbf{P}_{m}(I,k) =  \big\{ p  \in \mathbf{P}_{m}(I)  \;:\;    \omega _p\subset D  (I,k )
\big\}  
\end{equation*}  
is controlled by the number of  jumps $m$ makes on $D (I,k )$,  which is at most $ J$. Therefore, we can divide $ \mathbf{P}_{m}(I)$ into at most $ J$ collections 
$ \mathbf G ^{j}$, for $ 1\leq j \leq J$ so that $ \mathbf G ^{j } \cap  \mathbf{P}_{m}(I,k)$ has at most one element.  Therefore, for   $ \phi \in L ^2 (I)$, we have 
\begin{equation} \label{e:L2g}
\begin{split} 
\lvert  I\rvert ^{-1/2}   |\langle f_I , \phi  \rangle  | 
& = \left| \sum_{j=1} ^{J} \lvert  I\rvert ^{-1/2} \langle  f_I , P _{\mathbf G ^{j}} \phi \rangle  \right|
%= \left|\sum_{j=1} ^{J}  \langle  f_I , P _{\mathbf G ^{j}} \phi \rangle_{L^2(I)}   \right|
\\
& \lesssim  \sum_{j=1} ^{J}  \langle f_I \rangle _{I, \psi_2} \langle P _{\mathbf G ^{j}} \phi  \rangle _{I, \textup{exp} (L ^2 )} 
\\
& \lesssim   \sum_{j=1} ^{J} 
 \langle P _{\mathbf G ^{j}} \phi  \rangle _{I, 2} \lesssim \sqrt J \langle  \phi \rangle_{I,2} .  
\end{split}
\end{equation}
Here, we have denoted 
\[
P _{\mathbf G ^{j}} \phi=\sum_{p\in \mathbf G ^{j} } \langle\phi, w_p \rangle w_p,
\]
we have later  
 appealed to the $ \psi_2$-$\textup{exp} (L ^2 )$ duality,   applied the Zygmund inequality \eqref{e:Z} in its equivalent dual form and used orthogonality to obtain the last inequality.  This proves \eqref{e:fI} by duality. 
 
It remains to prove the square function estimate \eqref{e:SfI}.  By the $ J$ dependence in that inequality, it suffices to prove it for the case of $J=1 $, namely that the collections $ \mathbf P _{m} (I,k)$ have cardinality at most one for all integers $ k$.   Then, note that  for $ p \in \mathbf P _{m} (I,k)$, 
we can expand the projection relative $ w _p $ in terms of Haar functions:  
\begin{equation*}
\langle f , w _p \rangle w _p = \sum_{ \substack{Q\subset I_p\\ 2 ^{k} \lvert  Q\rvert = \lvert  I_p\rvert   }} \langle f , w _p  \rangle \langle w _p , h_Q \rangle h_Q . 
\end{equation*} 
The Haar functions are disjointly supported, and this estimate is trivial:
\begin{equation*}
\lvert  \langle w _p , h_Q \rangle \rvert \leq \frac{ \sqrt {\lvert  Q\rvert }} { \sqrt {\lvert  I_p\rvert }} . 
\end{equation*}
So, we can compute a component of the Haar square function of $ f_I$ as 
\begin{align*}
\sum_{ \substack{Q\subset I_p\\ 2 ^{k} \lvert  Q\rvert = \lvert  I_p\rvert   }} 
\bigl\lvert   \langle f , w _p  \rangle \langle w _p , h_Q \rangle h_Q \bigr\rvert ^2 
& 
\leq  \frac{\langle f , w _p  \rangle ^2 } {\lvert  I_p\rvert } 
\sum_{ \substack{Q\subset I_p\\ 2 ^{k} \lvert Q\rvert = \lvert  I_p\rvert   }}  
\lvert  Q\rvert  \cdot h_Q ^2 \lesssim  \frac{\langle f , w _p  \rangle ^2 } {\lvert  I_p\rvert } \mathbf 1_{I_p}.  
\end{align*}
It follows from \eqref{e:fI} that 
\begin{align*}
(S _{I} f_I ) ^2 & \lesssim  \mathbf 1_{I }\sum_{p \in \mathbf P _{m} (I)} \frac{\langle f_I , w _p  \rangle ^2 } {\lvert  I\rvert }  \lesssim  1 .  
\end{align*}
This proves \eqref{e:SfI}, and completes the Lemma. 
\end{proof}
%%%%%%%%%%%%%%%%%%%%%%%%%%%%%%

We need a variant of the previous argument.  In the hypothesis, we can impose conditions on 
 the  local $\langle f  \rangle _{I,q} $ norms, 
for $ 1 < q \leq  2$.  The conclusions then  depend  on the number of jumps $ J$ of the $ R _{q,1}$ atom in a different way. 
The precise form of this variant is as follows. 
%%%%%%%%%%%%%%%%%%%%%%%%%%%%%% LEMMA LEMMA LEMMA
\begin{lemma}\label{l:keyp} Let $1<q\leq 2$, $ f $ be supported on  $ I_0\in \mathcal D_0$, 
and let $ m $ be an  $ R _{q,1}$-atom of at most  $J$ jumps, adapted to $ I_0$. 
Let $I \in \mathcal I $ be a collection of disjoint dyadic subintervals  of $I_0$ with 
\[
\left(\sup_{ I_0\setminus\left( \bigcup_{I\in \mathcal I} I\right)} |f| \right) + \left(
\sup_{I \in \mathcal I } \langle f \rangle _{I, q}  \right) \leq C \langle f \rangle _{I_0, q} 
\]
%for any larger collection of intervals $ \mathcal I$,
for some given $C>0$. Then there exists a decomposition of $f$ as in \eqref{e:gb}, satisfying \eqref{e:gb}-\eqref{e:Th}, and with \eqref{e:fI} replaced by
\begin{equation} \label{e:gq}
\langle f_{I} \rangle _{I,2} \lesssim  {\frac{\lvert J\rvert ^{\frac1q-\frac12}}{\sqrt{q-1}}}   \langle f \rangle _{I_0, q} , \qquad I \in \mathcal I.
 \end{equation}
\end{lemma}
%%%%%%%%%%%%%%%%%%%%%%%%%%%%%% PROOF PROOF PROOF

%%%%%%%%%%%%%%%%%%%%%%%%%%%%%% SECTION  SECTION SECTION
%%%%%%%%%%%%%%%%%%%%%%%%%%%%%% SECTION  SECTION SECTION 
\section{Sparse Bounds for  Multipliers} \label{s:4}

In this section, we will prove the   estimate  
\begin{equation} \label{e:Two}
\lVert T _{m}\rVert _{ \psi _2, q} \lesssim \tfrac{\|m\|_{R_{q,1}}  }  {  \sqrt{q-1}}, \qquad  1< q \leq 2.  
\end{equation}
This is the inequality \eqref{e:Tq}. 
Due to the atomic nature of our multiplier spaces we may restrict ourselves to the case of $m$  being an $R_{q,1}$-atom with  at most $J$ jumps for some fixed but arbitrary $J \in \mathbb N$.  Recall that in this case $\|m\|_{R_{q,1}}=1 $.  
%To unify notation,  we write within this section that \[ \langle \phi  \rangle _{I, q}= \begin{cases} \langle \phi  \rangle _{I,\psi_2} &q=1 \\ \langle \phi  \rangle _{I, q} & 1<q \leq 2. \end{cases} \] We also assume that the multiplier $ T_m  $ satisfies $ T_m w_1 =0$.  This is a completely harmless assumption which is preserved within the recursion described below.

We need to verify that whenever   $1< q\leq 2$,   and 
 $ f$ and $ \phi $ are bounded functions on $ I_0 = [0,1)$,  there is a sparse collection of intervals $ \mathcal S$ 
so that 
\begin{equation*}
\lvert  \langle T_m f, \phi  \rangle\rvert \lesssim 
\frac{1}{ \sqrt{q-1}} 
\sum_{I\in \mathcal S}  \lvert  I\rvert 
\langle f \rangle _{I, \psi_2 }  \langle \phi  \rangle _{I,q} . 
\end{equation*}
%With this notation,  we need to verify that whenever   $1\leq q\leq 2$,   and 
% $ f$ and $ \phi $ are bounded functions on $ I_0 = [0,1)$,  there is a sparse collection of intervals $ \mathcal S$ 
%so that 
%\begin{equation*}
%\lvert  \langle T f, \phi  \rangle\rvert \leq C_q \sum_{I\in \mathcal S}  \lvert  I\rvert 
%\langle f \rangle _{I, \psi_2}  \langle \phi  \rangle _{I,q}, \qquad C_{q}= \begin{cases} C & q=1 \\ \frac{C}{\sqrt{q-1}} & 1<q\leq 2.\end{cases}
%\end{equation*}
The proof is recursive, and it suffices to prove the recursive step.  
We claim that there is a collection $ \mathcal I $ of disjoint subintervals $ I\subset I_0=[0,1]$, 
whose union has measure  at most $ \tfrac{1}2$, so that 
\begin{equation}\label{e:2show}
\lvert  \langle T_m f, \phi  \rangle\rvert 
\lesssim \frac1{\sqrt{q-1}}
\langle f \rangle _{I_0, \psi_2}  \langle \phi  \rangle _{I_0,q} 
+ \sum_{I\in \mathcal I}  \lvert  \langle T _{m_I} (f \mathbf 1_{I}) , \phi \mathbf 1_{I}  \rangle\rvert . 
\end{equation}
%\begin{equation}\label{e:2show}
%\lvert  \langle T f, \phi  \rangle\rvert 
%\leq C_q
%\langle f \rangle _{I_0, \psi_2}  \langle \phi  \rangle _{I_0,q} 
%+ \sum_{I\in \mathcal I}  \lvert  \langle T _{m_I} (f \mathbf 1_{I}) , \phi \mathbf 1_{I}  \rangle\rvert . 
%\end{equation}
Since the induced multipliers $ m_I$ are again $R_{q,1}$  atoms with at most $ J$ jumps adapted to $I$, 
one can recurse on the last sum above.  
We apply the key Lemma~\ref{l:key} to the collection $\mathcal I$ of maximal dyadic subintervals of $I_0$ such that
\[
\max\left\{ \frac{\langle f\rangle _{I,\psi_2} }{\langle f \rangle _{I_0,\psi_2}},  \frac{\langle\phi \rangle _{I,q} }{\langle \phi \rangle _{I_0,q}}\right\} > 4. 
\]
%We apply the key Lemma~\ref{l:key} to the collection $\mathcal I$ of maximal dyadic subintervals of $I_0$ such that
%\[
%\max\left\{\frac{\langle f \rangle _{I, \psi_2}}{\langle f \rangle _{I_0, \psi_2}},  \frac{\langle\phi \rangle _{I,q} }{\langle \phi \rangle _{I_0,q}}\right\} > 4. 
%\]
 From this, we get 
a collection of disjoint  intervals $ \mathcal I$ whose union is at most $ \tfrac{1}2$, 
so that the functions  $ f$,  and $ \phi  $, as well as the  multiplier $ m$, 
have the decompositions \eqref{e:gb} described in Lemma \ref{l:key}. 
We then expand 
\begin{align}  \label{e:EX1}
\langle T _{m} f,  \phi \rangle 
& = \langle T _{m} \tilde f,  \tilde \phi \rangle  
\\ \label{e:EX2}
& \quad + \sum_{I\in \mathcal I} \langle T _{m} \tilde f , \phi_{I }'  \rangle 
+ \sum_{I\in \mathcal I} \langle T _{m} f_I'  ,  \tilde  \phi   \rangle
\\ \label{e:EX3}
& \quad+ \sum_{I\in \mathcal I} 
 \sum_{I'\in \mathcal I} 
\langle T _{m} f_I' , \phi_{I'}' \rangle .  
\end{align}
Note that 
$
 \lVert T _{m}\rVert _{2 \to 2 } \leq \lvert  J\rvert ^{- \frac{1}q}
$. 
Combine this estimate with \eqref{e:ftilde} for $f$ and with the  $ q$ equivalent formulation of \eqref{e:ftilde},  namely 
\begin{equation*}
\langle \tilde g \rangle _{I_0, 2} \lesssim \frac{\lvert J\rvert  ^{\frac1{q} - \frac1{2}}} {\sqrt{q -1}} \langle g \rangle _{I_0, q}
\end{equation*}
for $g=\phi$,  to see that 
\begin{align} \label{e:Tmg}
\lvert  \langle T _{m} \tilde f  ,   \tilde \phi   \rangle \rvert  
& \leq 
\lVert T_m\rVert _{2\to 2} \langle  \tilde f\rangle _{I_0, 2} \langle \tilde \phi \rangle _{I_0, 2} \lvert  I_0\rvert 
\lesssim   
  \frac1{\sqrt {q-1}} 
\langle f \rangle _{I_0,   \psi _2 }  \langle \phi  \rangle _{I_0, q}  \lvert  I_0\rvert. 
\end{align}
This controls the term on the right hand side of  \eqref{e:EX1}.  
The two terms in \eqref{e:EX2} are dual to each other.  Recall that  $ T _{m} f'_I = T _{m_I} f'_I  
= T _{m_I} (f \mathbf 1_{I})$.  And, all functions are supported on $ I$. Therefore, 
\begin{equation*}
 \langle T _{m} f_I'  ,     \tilde \phi    \mathbf 1_{I_0 \setminus I} \rangle 
 =  \langle T _{m_I} f_I'  ,     \tilde \phi    \mathbf 1_{I_0 \setminus I}  \rangle =0
\end{equation*}
since the last two functions are disjointly supported.  We also have 
\begin{equation*}
 \langle T _{m} f_I'  ,     \tilde \phi    \mathbf 1_{  I} \rangle  
 =  \langle T _{m_I} f_I' , \phi _I \rangle =0 
\end{equation*}
since the last two functions are in the linear span of the tiles $ \mathbf P (I)$, but are supported 
on disjoint tiles.  

It remains to consider the terms in \eqref{e:EX3}. But this double sum diagonalizes to 
\begin{equation*}
 \sum_{I\in \mathcal I} 
\langle T _{m} f_I' , \phi_{I}' \rangle
=
 \sum_{I\in \mathcal I} 
\langle T _{m_I} (f \mathbf 1_{I}) , \phi \mathbf 1_{I} \rangle. 
\end{equation*}
And, this is the sum we recurse on, completing the proof of \eqref{e:2show}.

%%%%%%%%%%%%%%%%%%%%%%%%%%%%%% SECTION  SECTION SECTION
%%%%%%%%%%%%%%%%%%%%%%%%%%%%%% SECTION  SECTION SECTION 

\section{The Square Function Bounds} \label{s:6}
In this section, we first demonstrate the sparse bound \eqref{eq:Sm} for the Walsh-Littlewood-Paley square function $ S _{\lambda }$ as defined in \eqref{e:lam}.    At the end, we explain how this 
argument also proves \eqref{eq:ST}, the sparse bound for $ S_2 \circ T_m$.  

Let $ \{ \mu _k\}$ be any increasing sequence of powers of two. The square function 
\begin{equation*}
\Bigl[ \sum_{k} \lvert  T _{[ \mu _k , \mu _{k+1})} f\rvert^2  \Bigr] ^{1/2} 
\end{equation*}
is an example of a martingale square function, hence it satisfies a much stronger $ (1,1)$ sparse bound.  
(A pinpoint reference for this sparse bound does not seem to be readily available.  It is straight forward to modify 
 the proof for a $(1,1)$ sparse bound for martingale transforms in \cite{MR3625108}.)
 Subtracting  a choice of such a square function from $ S _{\lambda }$, we are left with a simpler 
problem.  

We phrase that problem here, with an eye to the recursive proof of the sparse bound.  
 Say that $ \Omega $ is a \emph{$ I$-good} collection of intervals if 
the intervals $ \omega \in \Omega $ are pairwise disjoint and take the 
the form   $\omega= [ 2 ^{k} , v_k )$, for some $ v_k < 2 ^{k+1}$ and some integer $k$ with $2^{k}|I| \in \mathbb N $.    
We shall prove:

%%%%%%%%%%%%%%%%%%%%%%%%%%%%%% PROPOSITION PROPOSITION PROPOSITION
\begin{proposition}\label{p:good} Fix $ 1\leq p \leq 2$.  For any $I_0$-good collection of intervals $ \Omega $, 
and all $ f, g \in L ^2 (I_0)$, 
there is a collection  $ \mathcal I$ of disjoint intervals $ I\subset I_0$ so that $ \bigcup  _{\mathcal I} I$ 
has measure at most $ \frac{1}2 \lvert  I_0\rvert $, and  
\begin{align}  \label{e:good}
\langle (S _{\Omega  } f) ^{p} , g  \rangle 
& \lesssim  \lvert  I_0\rvert \cdot  \langle f \rangle _{I_0, \psi _2} ^{p} \langle g \rangle _{I_0} + 
\sum_{I\in \mathcal I} \langle (S _{\Omega_I} f \mathbf 1_{I}) ^{p} , g \mathbf 1_{I}  \rangle . 
\end{align}
where $ \Omega _I = \{\omega ^{o} _I \;:\; \omega \in \Omega \}  $ is $ I$-good.   
\end{proposition}
%%%%%%%%%%%%%%%%%%%%%%%%%%%%%% PROPOSITION PROPOSITION  PROPOSITION

%%%%%%%%%%%%%%%%%%%%%%%%%%%%%% PROOF PROOF PROOF
\begin{proof} 
It is important to make this alternate description of $ S _{\Omega } f$. 
As $\Omega$ is $I_0$-good, we have $\Omega=\{[2^{k},v_k):k\in \mathbf{k}\} $ and the integers $k \in \mathbf{k} $ satisfy $2^k|I_0|\geq 1$.
Note that we have 
\begin{align}
(S _{\Omega } f ) ^2 & = \sum_{k  \in \mathbf k } \lvert  T _{[2 ^{k}, v_k)} f \rvert ^2 
\\ \label{e:rep1}
& =  \sum_{k  \in \mathbf k} \lvert  T _{[0, v_k -2 ^{k})} M _{w _{2 ^{k}}}f \rvert ^2 
\end{align}
where $ M _{\phi } f = \phi \cdot f$ is the multiplication operator.  That is, we can rescale the interval on which we are projecting by precomposing with a Walsh-Fourier multiplier.  

We carry this further. Recall the equivalence \eqref{e:WalshTile}.  For a collection of tiles $ \mathbf P_k$, we have 
\begin{equation} \label{e:==}
T _{[0, v_k -2 ^{k}) }  \phi = \sum_{p \in \mathbf P_k} \langle \phi , w_p \rangle w_p.  
\end{equation}
%Moreover, by \eqref{e:signWalsh},  each function $ w _{v _{k} -2 ^{k}+1} w _{p}$  is a signed Haar function.   
Given $ I\subset I_0$, we set $ \mathbf P _{k} (I) = \{p \in \mathbf P_k \;:\; \lvert  I_p\rvert < \lvert  I\rvert  \}$, 
and set $ \Omega (I) = \{ [ 2 ^{k}, v_k) \;:\; k \in \mathbf{k},\, 2 ^{k} \geq  \lvert  I\rvert ^{-1} \}$.

\bigskip 

Take $ \mathcal I$ to be the maximal intervals $ I\subset I_0$ so that $ \langle f \rangle _{I, \psi _2} > C 
\langle f \rangle _{I_0, \psi _2}$, or $ \langle g \rangle _{I} > C \langle g \rangle _{I_0}$. For $ C$ an absolute constant big enough the set $ E= \bigcup _{I\in \mathcal I } I$ has measure at most $\frac{1}4 \lvert  I_0\rvert $, 
and we can apply Lemma~\ref{l:key}.   
Applying Lemma \ref{l:key} to $\mathcal{I}$, one obtains a decomposition \[f= f _{\infty }+\sum_{I\in\mathcal{I}}f_I + f'_I\] satisfying $\text{supp }f_I,\ \text{supp }f_I'\subset I$ and 
\[
S_{\Omega } f'_I=S_{ \Omega _I} f'_I= \mathbf 1_{I}S_{ \Omega _I}(f\mathbf{1}_I),
\]
which immediately follows from (\ref{e:Th}) for each $T_{\omega _k }$. 
Recalling \eqref{e:ftilde}, and using subadditivity of $ x \to x ^{p/2}$ estimate 
\begin{equation}\label{e:Sf<}
(S _{\Omega } f )^p \leq (S _{\Omega } \tilde f )^p+ \sum_{I\in \mathcal I} (S _{\Omega_I } f'_I)^p . 
\end{equation}

Using a standard decomposition, we can write 
\[
 g = g  _{\infty }+\sum_{I\in\mathcal{I}} g  _I= g \mathbf{1}_{I_0\setminus \cup_{I\in\mathcal{I}}I}+\sum_{I\in\mathcal{I}} g \mathbf{1}_I,
\] where the terms above satisfy 
\[
\|  g  _{\infty }\|_{L^\infty}\lesssim \langle  g  _I\rangle_{I_0,1},\qquad 
\langle  g  _I\rangle_{I,1}\lesssim \langle  g \rangle_{I_0,1},\qquad \text{supp } g  _I\subset I.
\]
Combining these estimates,   we claim that 
\begin{align}  \label{e:SQ1} &
\lvert \langle (S_{\Omega }  \tilde f ) ^{p},   g  _{\infty } \rangle  \rvert\lesssim |I_0| \langle f\rangle_{I_0,\psi_2} ^{p}\langle  g \rangle_{I_0,1},
\\ &\label{e:SQ3}
\sum_{I\in\mathcal{I}}\lvert \langle (S_{\Omega } \tilde  f) ^{p}  ,  g  _I\rangle \rvert
\lesssim  \langle f\rangle_{I_0,\psi_2} ^{p}\sum_{I\in\mathcal{I}}|I|\langle  g  _I\rangle_{I,1}\lesssim \langle f\rangle_{I_0,\psi_2}^{p}\langle  g \rangle_{I_0,1} \lvert  I_0\rvert. 
\end{align}
Indeed, these two inequalities immediately complete the proof of the Proposition.    

\medskip 

The inequality \eqref{e:SQ1} is seen  as follows. Using H\"older's inequality, 
\begin{align*}
\lvert \langle (S_{\Omega }   \tilde f  ) ^{p},   g  _{\infty } \rangle  
& \leq   \lVert g _{ \infty }\rVert _{ (2/p)'  }
\lVert S _{\Omega } \tilde f\rVert _{2} ^{p}
\end{align*}
which completes the proof of \eqref{e:SQ1} in view of the obvious $ L ^2 $ estimate for $ S _{\Omega }$ and \eqref{e:ftilde}.  

The second inequality \eqref{e:SQ3} depends upon this essential point: For each $ I\in \mathcal I$, the function $ S _{\Omega } \tilde f$ is constant on $ I$.  

To see this, return to the representation \eqref{e:rep1} and \eqref{e:==}. 
For each $ k \in \mathbf k$, we have 
\begin{equation*}
T _{[0, v_k -2 ^{k}) }  M _{w _{2 ^{k}}} \tilde f = \sum_{p \in \mathbf P_k} \langle  M _{w _{2 ^{k}}} \tilde f , w_p \rangle w_p.  
\end{equation*}
Fix $ I\in \mathcal I$.  If $ p \in \mathbf P_k$ with $ I_p \subset I$, it follows from construction of $ \tilde f$ that the inner product $ \langle  M _{w _{2 ^{k}}} \tilde f , w_p \rangle$ is zero.  And, if $ I\subsetneq I_p$, it follows that 
function $ w _{v _{k} -2 ^{k}} w _{p}$  is a signed Haar function.   That is, the function above is constant on $ I$. 
With this claim established, note that \eqref{e:SQ3} then follows from the same reasoning for \eqref{e:SQ1}.  
\end{proof}
%%%%%%%%%%%%%%%%%%%%%%%%%%%%%% PROOF PROOF PROOF
 
%%%%%%%%%%%%%%%%%%%%%%%%%%%%%% PROOF PROOF PROOF
\begin{proof}[Proof of \eqref{eq:ST}]  
To  prove this estimate for $ S_2 \circ T_m$, it suffices to restrict attention to a multiplier $ m$ 
of the form $ \sum_{k} \sigma _k \mathbf 1_{[2 ^{k}, \nu _k)}$ for some $ 2 ^{k}< \nu _k \leq 2 ^{k+1}$, and $ \sigma _k \in \{-1,0,1\}$. 
Then, $ S_2 \circ T_m$ is no more than the square function considered in Proposition~\ref{p:good}. 
We see that the claimed sparse bound holds.  
\end{proof}
%%%%%%%%%%%%%%%%%%%%%%%%%%%%%% PROOF PROOF PROOF

%%%%%%%%%%%%%%%%%%%%%%%%%%%%%% SECTION  SECTION SECTION
%%%%%%%%%%%%%%%%%%%%%%%%%%%%%% SECTION  SECTION SECTION 
\section{Proof of   Corollary \ref{c:lerner}} \label{s:weight}

We will use in what follows the by now widespread language of $A_p$, $A_\infty$ and Reverse H\"older classes; for the relevant definitions, see for instance \cite{MR3455749,MR3563275} and references therein.

%%%%%%%%%%%%%%%%%%%%%%%%%%%%%% SUBSECTION SUBSECTION SUBSECTION SUBSECTION
 %%%%%%%%%%%%%%%%%%%%%%%%%%%%%% SUBSECTION SUBSECTION SUBSECTION SUBSECTION 
\subsection{Proof of the $ L (\log L) ^{1/2} (w) \to  L ^{1, \infty } (w)$ Bound}%\label{ss.}

We prove \eqref{e:cP2}.  This argument is known, and we do not seek to make this argument quantitative.  
By the restricted weak type approach, given $ f\in L (\log L ) ^{1/2} (w)$, and  $ G \subset [0,1]$, we select $ G'\subset G$ with $ 2w (G') > w (G)$ so that 
for any $ \lvert  g\rvert = \mathbf 1_{G'} $
\begin{equation}\label{e:cP21}
\lvert \langle T_m f, g w\rangle\rvert \lesssim C _{[w] _{A_1}} \lVert f\rVert _{L (\log L) ^{1/2} (w)}. 
\end{equation}

We can assume that $  \lVert f\rVert _{L (\log L) ^{1/2} (w)} =1$.  Define a  weighted Orlicz maximal function 
\begin{equation*}
M _{\psi _2 } ^{w} f = \sup _{I} \mathbf 1_{I}  \langle f \rangle _{I, \psi _2} ^{w}. 
\end{equation*}
The superscript $ w$ indicates that we integrate with respect to $ w $ measure: 
$ \langle f \rangle _{I, \psi _2} ^{w} =   \lVert f \mathbf 1_{I} \rVert _{\psi _2 (L) (\frac{w (dx)} {w (I)})}$. 
It is straight forward to verify that 
\begin{equation*}
\lVert M _{\psi _2} ^{w} f\rVert _{L ^{1, \infty }   (w)} \lesssim  \lVert f\rVert _{L(\log L) ^{1/2} (w)}. 
\end{equation*}
Indeed, given threshold $ t >1$, let $ \mathcal I $ be the maximal dyadic intervals with $ \langle f \rangle _{L (\log L) ^{1/2} } > t$. Then, 
\begin{align*}
w (M _{\psi _2} ^wf >t )  &\lesssim  \sum_{I\in \mathcal I} \int _{I} \psi _2 ( f /t) \;w(dx) 
\\
& = \sum_{I\in \mathcal I} \int _{I} \lvert  f/t\rvert \sqrt {\log_+ f/t} \; w (dx)  
\lesssim t ^{-1} 
\langle f \rangle _{L (\log L) ^{1/2}  (w)}. 
\end{align*}

\smallskip 

Then, set $ G' = G \setminus \{ M _{\psi _2} ^{w}f  > K/   w (G)\}$.  
For a choice of absolute constant $ K$, we have $ 2 w (G') > w (G)$.  
Further, there is a choice of $ q_0 -1 \simeq [w] _{A_1} ^{-1} $ for which $ w \in RH_ {q_0}$, in particular 
\begin{equation*}
\langle w \rangle _{I, q_0} \leq 2 \langle w \rangle_I .  
\end{equation*}
Let $ q = (1 + q_0)/2$.  
Apply the $ \psi _2$-$L^q$ sparse bound \eqref{e:Tq} with this choice of $ q$. Then, 
\begin{align} \label{e:cP22}
\langle T_mf  , g  w \rangle  & \lesssim [w] _{A_1}  ^{1/2} 
\sum_{I\in \mathcal I} \langle f  \rangle _{I, \psi _2} \langle g w \rangle _{I,q} \lvert  I\rvert 
\end{align} 

The individual summands are addressed this way.  First, by the definition of $ A_1$, 
\begin{equation}\label{e:cP23}
\langle f  \rangle _{I, \psi _2}  \lesssim 
\frac{1} {\lvert  I\rvert } \int _{I} \psi _2 (f ) \; dx \cdot \frac{w (I)} {w (I)} 
\leq [w] _{A_1} \langle f  \rangle _{I, \psi _2} ^{w}.  
\end{equation}
Second, we claim that  for some $ 0 < \rho = \rho _{[w] _{A_1}} < 1$, there holds 
\begin{equation}\label{e:cP24}
  \langle g w \rangle _{I,q}  \lvert  I\rvert \lesssim  C _{[w] _{A_1}}  (\langle g \rangle _{I } ^{w }) ^ \rho  w (I).    
\end{equation}
Indeed,  note that the reverse H\"older property and choice of $ q$ imply that the weight $ w ^{ q-1}$ is in the class $ A _{\infty } (w)$, 
the latter being the Muckenhoupt class $ A _{\infty }$ relative to weight $ w$.  It follows that $ w ^{q}$ is quantitatively absolutely continuous 
with respect to $ w$.  That is, for a choice of $ \rho$ as in \eqref{e:cP24}, there holds 
\begin{equation*}
\langle  g \rangle _{I,q} ^{w ^{q}} \lesssim ( \langle g \rangle _{I} ^{w} ) ^{\rho }.  
\end{equation*}
Then, appealing to the reverse H\"older property again,  we have 
\begin{align*}
 \langle g w \rangle _{I,q}  \lvert  I\rvert  & = \langle  g \rangle _{I,q} ^{w ^{q}}  \langle w  \rangle _{I,q}\lvert  I\rvert 
 \\
 & \lesssim     ( \langle g \rangle _{I} ^{w} ) ^{\rho }  \langle w  \rangle _{I}\lvert  I\rvert  =  \langle  g \rangle _{I,q} ^{w ^{q}}   w (I). 
\end{align*}
This proves \eqref{e:cP24}.  

\smallskip 
Combine \eqref{e:cP23} and \eqref{e:cP24} so see that 
\begin{align} \label{e:cP25}
\eqref{e:cP21} 
&\lesssim    C _{[w] _{A_1}} \sum_{I\in \mathcal I} \langle f  \rangle _{I, \psi _2} ^{w}
\langle g \rangle _{I } ^{w } w (I) .
\end{align}
Turn to a pigeonholing argument.  
Divide the collection $ \mathcal I$ into collections $ \mathcal I _{j,k}$, for $ j,k\geq 1$.  An interval $ I\in \mathcal I _{j,k}$ by the conditions  
\begin{gather*}
2 ^{-j} K /w (G) <  \langle f  \rangle _{I, \psi _2} ^{w} \leq 2 ^{-j+1} K /w (G), 
\\
\textup{and} \quad 
2 ^{-k}   <  \langle g  \rangle_{I } ^{w } \leq 2 ^{-k+1}  . 
\end{gather*}
Observe that 
\begin{equation*}
\sum_{I\in \mathcal I _{j,k}} w (I) 
\lesssim [w] _{A _{\infty }}\sum_{I\in \mathcal I _{j,k} ^{\ast} } w (I) 
\lesssim \min \{  2 ^{j} ,  2 ^{  k } \} w (G). 
\end{equation*}
We then conclude that 
\begin{equation*}
\eqref{e:cP25} \lesssim \sum_{j,k \geq 1}  2 ^{-j- \rho k}\min \{  2 ^{j} ,  2 ^{  k } \} 
\lesssim C _{[w] _{A_1}} .   
\end{equation*}
And this completes the proof.
%%%%%%%%%%%%%%%%%%%%%%%%%%%%%% SUBSECTION SUBSECTION SUBSECTION SUBSECTION
 %%%%%%%%%%%%%%%%%%%%%%%%%%%%%% SUBSECTION SUBSECTION SUBSECTION SUBSECTION 
\subsection{Proof of the $ L ^{p}  (w) \to L ^{p} (w)$ for $ R _{q,1}$ Multipliers}%\label{ss.}
The inequalities \eqref{e:Tqq} are a direct consequence of our sparse bound for $ R _{q,1}$ multipliers and bounds for sparse operators.  
In particular, we need only cite  \cite{MR3531367}*{Prop. 6.4}.  The interested reader can easily track quantitative dependence on characteristics for the weight.  We suppress the details.

%%%%%%%%%%%%%%%%%%%%%%%%%%%%%% SUBSECTION SUBSECTION SUBSECTION SUBSECTION
 %%%%%%%%%%%%%%%%%%%%%%%%%%%%%% SUBSECTION SUBSECTION SUBSECTION SUBSECTION 
\subsection{Proof of the Square Function Bound }%\label{ss.}
It suffices to prove the weighted inequalities in \eqref{e:cS} for $ p=5/2$, 
as the full sharp range may be obtained by extrapolation, as   noted by A.~Lerner in \cite{180306981}.  

Fix bounded functions $ f,g$, and let $ w \in A _{5/2}$.
We should prove 
\begin{equation}\label{e:SM}
\langle  (S _{\Omega } f \sigma ) ^{3/2}, g w \rangle \lesssim [w] _{A_p} ^{3/2} 
\lVert f\rVert _{L ^{5/2} ( \sigma)} ^{ 3/2} \lVert g\rVert _{L ^{5/2} (w )}.  
\end{equation}
where $ \sigma = w ^{1- p'} = w^{-3/2}$ is the dual measure to $ w$.

Recall that $ \sigma \in A _{5/3}$, and $ [\sigma ] _{A _{5/3}} = [w] _{A_p} ^{2/3}$.  
Fix $ q -1 = c [w] _{A_p} ^{- 2/3} $, for a sufficiently small constant $ c$.  In particular, $ q$ is close enough to $ 1$ so that 
\begin{equation}  \label{e:RH1}
\langle \sigma  \rangle _{I, 1 + (q-1)(5/2q)'} \leq  
\langle \sigma  \rangle _{I, q} \leq 2 \langle \sigma  \rangle_I,
\end{equation}
We apply the sparse bound \eqref{eq:Sm}, although that inequality is stated in its $ \psi _2$ version, and we rather use the $ L ^{q}$ version:  For a sparse collection of intervals $ \mathcal I$, 
\begin{align} \label{e:SM1}
\langle  (S _{\Omega } f \sigma  ) ^{3/2}, g w  \rangle 
\lesssim   (q-1)^{- \frac{3}4}
\sum_{I\in \mathcal I} \langle  \sigma f  \rangle _{I,q} ^{3/2} \langle g w \rangle _{I} \cdot \lvert  I\rvert 
\end{align}
Note that $  (q-1)^{- \frac{3}4} \lesssim [w] _{A_p} ^{ \frac{1}2}$.  We continue with the analysis of the sum above.

Below, we will take averages of $ f$ relative to the weight $ \sigma ^{q}$, and employ a standard sparse trick. 
Note that we gain an additional power of $ [w] _{A_ {5/2}}$, giving us the $ 3/2$ power in \eqref{e:SM}.  
\begin{align*}
\sum_{I\in \mathcal I} \langle  f  \sigma  \rangle _{I,q} ^{3/2} \langle g w\rangle _{I} \cdot \lvert  I\rvert  
& = 
\sum_{I\in \mathcal I}( \langle f  \rangle _{I,q} ^{\sigma ^{q}}) ^{3/2} 
\langle \sigma  \rangle_{I,q} ^{3/2} \langle w \rangle_I
\langle g \rangle _{I} ^{w}\cdot \lvert  I\rvert    \qquad \textup{(Insert new averages)}
\\
& \lesssim 
\sum_{I\in \mathcal I}( \langle f  \rangle _{I,q} ^{\sigma ^{q}}) ^{3/2} 
\langle \sigma  \rangle_{I} ^{3/2} \langle w \rangle_I
\langle g \rangle _{I} ^{w}\cdot \lvert  I\rvert    \qquad  \textup{(By \eqref{e:RH1})}
\\
& \lesssim [w] _{A _{5/2}} 
\sum_{I\in \mathcal I}( \langle f  \rangle _{I,q} ^{\sigma ^{q}}) ^{3/2} 
\langle g \rangle _{I} ^{w}\cdot \lvert E_  I\rvert    \qquad \textup{(Defn. of $ A_ {5/2}$)}
\\
& \leq [w] _{A _{5/2}} \int 
\sum_{I\in \mathcal I}( \langle f  \rangle _{I,q} ^{\sigma ^{q} }) ^{3/2} 
\langle g \rangle _{I} ^{w}\cdot \mathbf 1_{E_I} \cdot  \sigma ^{3/5} w ^{2/5} \;dx  
\\
&\leq [w] _{A _{5/2}} 
\int   (M ^{\sigma ^{q}}  \lvert  f\rvert ^{q} ) ^{3/2q} \sigma ^{3/5}  \times M ^{w} g \cdot w ^{2/5} \; dx  
\\ 
& \leq  [w] _{A _{5/2}}  
\lVert  (M ^{\sigma ^{q}}  \lvert  f\rvert ^{q}) ^{3/2q} \rVert _{L ^{5/3} (\sigma )} 
\lVert M ^{w}g \rVert _{L ^{5/2} (w)}  \qquad \textup{(H\"older)}
\\ & 
\lesssim 
 [w] _{A _{5/2}}  \lVert  M ^{\sigma  ^{q}}  \lvert  f\rvert ^{q}  \rVert _{L ^{5/2q} (\sigma )} ^{3/2q}  \lVert g\rVert _{L ^{5/2} (w)}. 
\end{align*}
In the last line, it  is essential to note that on the function $ g$ we have used the `universal' or `martingale' maximal function estimate.  But, we have no such immediate recourse for the function $ f$.  

\smallskip 

This inequality will complete the proof of \eqref{e:SM}.  
\begin{equation}\label{e:SM2}
 \lVert  M ^{\sigma  ^{q}}  \phi  \rVert _{L ^{5/2q} (\sigma )}   \lesssim \lVert \phi \rVert _{L ^{5/2q} (\sigma )}. 
\end{equation}
But note that we are taking the $ \sigma ^{q}$ weighted dyadic maximal function on the left. 
To prove this inequality, we use the theory of $ A_p$ weights, relative to a non-Lebesgue measure. 
The relative measure is $ \sigma ^{q}$.  The definition of $ A_{5/2q} (\sigma  ^{q})$, given 
weight $ v (x) d \sigma ^{q}$, is 
\begin{equation*}
[v] _{A_{5/2q} (\sigma ^{q})} = \sup _{I} \langle v \rangle ^{\sigma ^{q}}_I \bigl[  \langle v ^{1- (5/2q)'} \rangle _{I} ^{\sigma ^{q}} \bigr] ^{5/2q-1}.  
\end{equation*}
Above, $ v ^{1-(5/2q)'}$ is the `dual weight.'  
In our case, the weight is $ \sigma ^{1-q}$.  The `dual weight' is $ \mu = \sigma ^{(1-q) (1-(5/2q)')}$. 
We check that $ v \in A _{5/2q} (\sigma ^{q})$, and that moreover, its $ A_{5/2q} (\sigma ^{q})$ characteristic is bounded by a constant.  By the martingale version of the Muckenhoupt theorem, we see that \eqref{e:SM2} holds. 

\smallskip 

For fixed interval $ I$, the term below should be bounded by a constant: 
\begin{align*}
\langle \sigma ^{1-q} \rangle ^{\sigma ^{q}}_I \bigl[  \langle  \mu  \rangle _{I} ^{\sigma ^{q}} \bigr] ^{5/2q -1 }
& = 
\frac{ \sigma (I)} {\sigma ^{q} (I)} \cdot 
\Bigl[ \frac{\mu \cdot \sigma ^{q} (I)} {\sigma ^{q} (I)} \Bigr] ^{5/2q-1} 
\\
& = \frac{ \sigma (I)  [ \mu \cdot \sigma ^{q} (I) ] ^{5/2q-1}}   
{ [\sigma ^{q} (I) ] ^{5/2q}}
\\
& = \frac{ \sigma (I)  [   \sigma ^{1 + (q-1)(5/2q)'} (I) ] ^{5/2q-1}}   
{ [\sigma ^{q} (I) ] ^{5/2q}}
\\
& \lesssim 
\frac{ \sigma (I)  \{   \sigma(I) ^{1 + (q-1)(5/2q)'} \lvert  I\rvert ^{(1-q)(5/2q)'} \} ^{5/2q-1}}   
{ [\sigma   (I) ^{q} \lvert  I\rvert ^{1-q}  ] ^{5/2q}}   
= 1. 
\end{align*}
In the last line, we have applied our reverse H\"older inequalities, \eqref{e:RH1}. 
Last of all,  equality follows by inspection.  

 %%%%%%%%%%%%%%%%%%%%%%%%%%%%%% SUBSECTION SUBSECTION SUBSECTION SUBSECTION 
\subsection{Proof of the $ L ^{p} (w) \to L ^{p} (w)$ Bound for Marcinkie\-wicz Multipliers}%\label{ss.}

As is noted by Lerner \cite{180306981}*{Conjecture 5.2}, it is enough to prove 
\eqref{e:cLp} in the case of $ p= \frac{5}2$, as an extrapolation argument (recalled in \cite{180306981}*{Thm. 2.8})
gives the desired bound for $ \frac{5}2 \leq p < \infty $, and then use duality to deduce the range 
$ 1 < p \leq \frac{5}3$.

We need only show that 
\begin{equation*}
\lVert T_m\rVert _{L ^{5/2} (w) \to L ^{5/2} (w)} \lesssim [w] ^{3/2} _{A_{5/2}} \lVert m\rVert _{M}
\end{equation*}
But this is easy to do, following an argument of Lerner.  
Appeal to a weighted inequality of Wilson \cite{MR972707,MR2359017} to  estimate 
the weighted norm of $ T_m$ in terms of $ S_2 \circ T_m$:  
\begin{align*}
\lVert T_m\rVert _{L ^{5/2} (w) \to L ^{5/2} (w)}  
& \lesssim [w] _{A _{5/2}} ^{\frac{1}2} 
\lVert  S_2 \circ T_m\rVert _{L ^{5/2} (w) \to L ^{5/2} (w)}  . 
\end{align*}
The latter operator satisfies the same sparse bounds as does the square functions $ S _{\lambda }$, 
as we see by comparing \eqref{eq:Sm} and  \eqref{eq:ST}.  
In particular, by the proof of \eqref{e:cS}, which we just completed, we see that 
\begin{equation*}
\lVert  S_2 \circ T_m\rVert _{L ^{5/2} (w) \to L ^{5/2} (w)}  \lesssim 
[w]  _{A_{5/2}} \lVert m\rVert _{M}. 
\end{equation*}
So our proof is complete.

%%%%%%%%%%%%%%%%%%%%%%%%%%%%%% SECTION  SECTION SECTION
%%%%%%%%%%%%%%%%%%%%%%%%%%%%%% SECTION  SECTION SECTION 
\section{Proof of the Lower Bound on Sparse Forms} \label{s:lower}

We prove \eqref{e:reverse}.  

%%%%%%%%%%%%%%%%%%%%%%%%%%%%%% PROPOSITION PROPOSITION PROPOSITION
\begin{proposition}\label{p:lerner}   There is a Marcinkiewicz multiplier $ m$,  so that for $ 1< q  < 2$,  there is a pair of functions $ f, g$ so that 
\begin{equation*}
\lvert  \langle T_m f,g \rangle\rvert  \gtrsim \frac{1} {q-1} \sup _{\mathcal S} \sum _{I\in \mathcal S} \lvert  I\rvert \langle f \rangle_{I,q} \langle g \rangle _{I,q} . 
\end{equation*}
The supremum is over all sparse collections of intervals $ \mathcal S$. 
\end{proposition}
%%%%%%%%%%%%%%%%%%%%%%%%%%%%%% PROPOSITION PROPOSITION PROPOSITION

%%%%%%%%%%%%%%%%%%%%%%%%%%%%%% PROOF PROOF PROOF
\begin{proof}
The Marcinkiewicz multiplier is 
\begin{equation*}
T_m f = \sum_{k=1} ^{\infty } \langle f, w _{2 ^{k}} \rangle w _{2 ^{k}}. 
\end{equation*}
 Recall that $ \{w _{2 ^{k}} \;:\; k\geq 1 \}$  is a Rademacher sequence.  
With $ q-1 \simeq 1/n$, for integer $ n$,  take  $ f = 2 ^{n}\mathbf 1_{[0, 2 ^{-n})}$,  from which we see that 
 that $ T_m f = \sum_{k=1} ^{n}  w _{2 ^{k}}$.  
 
 Take $ g $ to be the indicator of $ \{T_m f = -n\} $, which is the interval $ (1- 2 ^{-n},1]$.  Then, $ \langle T_m f, g \rangle = - n 2 ^{-n}$. 
The functions  $ f$ and $ g$ are supported on opposite ends of $ [0,1]$.   That means that 
the only sparse collection we need consider is $ \mathcal S = \{[0,1]\}$. 
Note that 
\begin{equation*}
2 ^{n} \lVert g\rVert_q = \lVert f\rVert _{q} = \bigl[ 2 ^{nq -n} \bigr] ^{1/q} \simeq 1. 
\end{equation*}
That is, we have $ \lvert  \langle T_m f,g \rangle\rvert \gtrsim (q-1) ^{-1} \lVert f\rVert_q \lVert g\rVert_q $. 
This completes the proof. 
\end{proof}
%%%%%%%%%%%%%%%%%%%%%%%%%%%%%% PROOF PROOF PROOF

%\bibliography{sparse,walsh_marc}	
\bibliographystyle{alpha,amsplain}


\end{document}